\documentclass[11pt,a4paper,oneside]{amsart}%
\usepackage{amsfonts}
\usepackage{amsmath}
\usepackage{amssymb}
\usepackage{graphicx}%
\providecommand{\U}[1]{\protect\rule{.1in}{.1in}}
\newtheorem{theorem}{Theorem}
\theoremstyle{plain}

\newtheorem{definition}{Definition}
\newtheorem{example}{Example}

\newtheorem{lemma}{Lemma}

\numberwithin{equation}{section}
\begin{document}
\title[ ]{Periodic and Asymptotically Periodic Solutions of Systems of Nonlinear
Difference Equations with Infinite Delay}
\author{Murat Ad\i var}
\address[M. Ad\i var]{ Izmir University of Economics\\
Department of Mathematics, 35330, Izmir Turkey}
\email{murat.adivar@ieu.edu.tr}
\urladdr{}
\author{H. Can Koyuncuo\u{g}lu}
\address[ H. C. Koyuncuo\u{g}lu]{ Izmir University of Economics\\
Department of Mathematics, 35330, Izmir Turkey}
\email{can.koyuncuoglu@ieu.edu.tr}
\urladdr{}
\author{Youssef N. Raffoul}
\address[Y. N. Raffoul]{ University of Dayton\\
Department of Mathematics, Dayton, OH 45469-2316, USA}
\email{yraffoul1@udayton.edu}
\urladdr{}
\thanks{}
\date{March 18, 2013}
\dedicatory{The final version will be published in Journal of Difference Equations and Applications.}
\begin{abstract}
In this work we study the existence of periodic and asymptotically periodic
solutions of a system of nonlinear Volterra difference equations with infinite
delay. By means of fixed point theory, we furnish conditions that guarantee
the existence of such periodic solutions.

\end{abstract}
\maketitle

\section{Introduction}

Consider the system of nonlinear Volterra difference equations with infinite
delay%
\begin{equation}
\left\{
\begin{array}
[c]{c}%
\Delta x_{n}=h_{n}x_{n}+\sum\limits_{i=-\infty}^{n}a_{n,i}f(y_{i})\\
\Delta y_{n}=p_{n}y_{n}+\sum\limits_{i=-\infty}^{n}b_{n,i}g(x_{i})\text{\ }%
\end{array}
\right.  , \label{1.1}%
\end{equation}
where $f$ and $g$ are real valued and continuous functions, and $\left\{
a_{n,i}\right\}  $, $\left\{  b_{n,i}\right\}  $, $\left\{  h_{n}\right\}  $,
and $\left\{  p_{n}\right\}  $ are real sequences. In this study, we use
\emph{Schauder's fixed point theorem }to provide sufficient conditions
guaranteeing the existence of periodic and asymptotically periodic solutions
of the system (\ref{1.1}). Since we are seeking the existence of periodic
solutions it is natural to ask that there exists a least positive integer $T$
such that%
\begin{equation}
h_{n+T}=h_{n}\text{,\ }p_{n+T}=p_{n}\text{,} \label{1.2}%
\end{equation}%
\begin{equation}
a_{n+T,i+T}=a_{n,i}, \label{1.2.1}%
\end{equation}
and%
\begin{equation}
b_{n+T,i+T}=b_{n,i} \label{1.2.2}%
\end{equation}
hold for all $n\in\mathbb{N}$, where $\mathbb{N}$ indicates the set of all
nonnegative integers.

Recently, there has been a remarkable interest in the study of Volterra
equations due to their applications in numerical analysis and biological
systems, see e.g. \cite{ab}, \cite{hb}, and \cite{cs}. There is a vast
literature on this subject in the continuous and discrete cases. For
instance,\ in \cite{[2]} the authors considered the two dimensional system of
nonlinear Volterra difference equations%
\[
\left\{
\begin{array}
[c]{c}%
\Delta x_{n}=h_{n}x_{n}+\sum\limits_{i=1}^{n}a_{n,i}f(y_{i})\\
\Delta y_{n}=p_{n}y_{n}+\sum\limits_{i=1}^{n}b_{n,i}g(x_{i})
\end{array}
\right.  ,\ \ \ \ n=1,2,...
\]
and classified the limiting behavior and the existence of its positive
solutions with the help of fixed point theory. Also, the authors of
\cite{[xy]} analyzed the asymptotic behavior of positive solutions of second
order nonlinear difference systems, while the authors of \cite{[wy]} studied
the classification and the existence of positive solutions of the system of
Volterra nonlinear difference equations. Periodicity of the solutions of
difference equations has been handled by \cite{[1]}, \cite{[3]}-\cite{[5]}. In
\cite{[6]} and \cite{[7]}, the authors focused on a system of Volterra
difference equations of the form
\[
x_{s}(n)=a_{s}(n)+b_{s}(n)x_{s}(n)+\sum_{p=1}^{r}\sum_{i=0}^{n}K_{sp}%
(n,i)x_{p}(i),\ \ n\in\mathbb{N},
\]
where $a_{s}$, $b_{s}$, $x_{s}:\mathbb{N}\rightarrow\mathbb{R}$ and
$K_{sp}:\mathbb{N}\times\mathbb{N}\rightarrow\mathbb{R}$, $s=1,2,...,r$, and
$\mathbb{R}$ denotes the set of all real numbers and obtained sufficient
conditions for the existence of asymptotically periodic solutions. They had to
construct a mapping on an appropriate space and then obtain a fixed point.
Furthermore, in \cite{[8]} \ the authors investigated the existence of
periodic and positive periodic solutions of system of nonlinear Volterra
integro-differential equations. The paper \cite{el} of Elaydi, was one of the
first to address the existence of periodic solutions and the stability
analysis of Volterra difference equations. Since then, the study of Volterra
difference equations has been vastly increasing. For instance, we mention the
papers \cite{kr}, \cite{md}, and the references therein. In addition to
periodicity we refer to \cite{kc} and \cite{mdi} for results regarding
boundedness. \newline The main purpose of this paper is to extend the results
of the above mentioned literature by investigating the possibility of
existence of periodic and the asymptotic periodic solutions for systems of
nonlinear Volterra difference equations with infinite delay.

Denote by $\mathbb{Z}$ and $\mathbb{Z}^{-}$ the set of integers and the set of
nonpositive integers, respectively. By a solution of the system (\ref{1.1}) we
mean a pair of sequences $\left\{  \left(  x_{n},y_{n}\right)  \right\}
_{n\in\mathbb{Z}}$ of real numbers which satisfies (\ref{1.1}) for all
$n\in\mathbb{N}$. The initial sequence space for the solutions of the system
(\ref{1.1}) can be constructed as follows. Let $S$ denote the nonempty set of
pairs of all sequences $\left(  \eta,\zeta\right)  =\left\{  \left(  \eta
_{n},\zeta_{n}\right)  \right\}  _{n\in\mathbb{Z}^{-}}$ of real numbers such
that
\[
\max\left\{  \sup_{n\in\mathbb{Z}^{-}}\left\vert \eta_{n}\right\vert
,\sup_{n\in\mathbb{Z}^{-}}\left\vert \zeta_{n}\right\vert \right\}  <\infty,
\]
and for each $n\in\mathbb{N}$ the series%
\[%
{\displaystyle\sum\limits_{i=-\infty}^{0}}
a_{n,i}f(\eta_{i})\text{ and }%
{\displaystyle\sum\limits_{i=-\infty}^{0}}
b_{n,i}g(\zeta_{i})
\]
converge. It is clear that for any given pair of initial sequences $\left\{
\left(  \eta_{n},\zeta_{n}\right)  \right\}  _{n\in\mathbb{Z}^{-}}$ in $S$
there exists a unique solution $\left\{  \left(  x_{n},y_{n}\right)  \right\}
_{n\in\mathbb{Z}}$ of the system (\ref{1.1}) which satisfies the initial
condition%
\begin{equation}
\left(
\begin{array}
[c]{c}%
x_{n}\\
y_{n}%
\end{array}
\right)  =\left(
\begin{array}
[c]{c}%
\eta_{n}\\
\zeta_{n}%
\end{array}
\right)  \text{ for }n\in\mathbb{Z}^{-}, \label{initial condition}%
\end{equation}
such solution $\left\{  \left(  x_{n},y_{n}\right)  \right\}  _{n\in
\mathbb{Z}}$ is said to be the solution of the initial problem (\ref{1.1}%
-\ref{initial condition}). For any pair $\left(  \eta,\zeta\right)  \in S$,
one can specify a solution of (\ref{1.1}-\ref{initial condition}) by denoting
it by $\left\{  \left(  x_{n}\left(  \eta\right)  ,y_{n}\left(  \zeta\right)
\right)  \right\}  _{n\in\mathbb{Z}}$, where%
\[
\left(  x_{n}\left(  \eta\right)  ,y_{n}\left(  \zeta\right)  \right)
=\left\{
\begin{array}
[c]{ll}%
\left(  \eta_{n},\zeta_{n}\right)  & \text{for }n\in\mathbb{Z}^{-}\\
\left(  x_{n},y_{n}\right)  & \text{for }n\in\mathbb{N}%
\end{array}
\right.  .
\]
In our analysis, we apply a fixed point theorem to general operators over a
Banach space of bounded sequences defined on the whole set of integers. Unlike
the above mentioned literature that dealt with stability of delayed difference
systems, in the construction of our existence type theorems we neglect the
consideration of phase space, for simplicity. For a similar approach we refer
to \cite{burton}.

We end this section by recalling the fixed point theorem that we use in our
further analysis.

\begin{theorem}
\label{thm1}\emph{(Schauder's fixed point theorem)}Let $X$ be a Banach Space.
Assume that $K$ is a closed, bounded and convex subset of $X$. If
$T:K\rightarrow K$ is a compact operator, then it has a fixed point in$\ K.$
\end{theorem}

\section{Periodicity}

In this section, we use Schauder's fixed point theorem to show that system
(\ref{1.1}) has a periodic solution.

Let $P_{T}$ be the set of all pairs of sequences $\left(  x,y\right)
=\left\{  \left(  x_{n},y_{n}\right)  \right\}  _{n\in\mathbb{Z}}$ satisfying
$x_{n+T}=x_{n}$ and $y_{n+T}=y_{n}$ for all $n\in\mathbb{N}$. Then $P_{T}$ is
a Banach space when it is endowed with the maximum norm%
\[
\left\Vert (x,y)\right\Vert :=\max\left\{  \max_{n\in\left[  1,T\right]
_{\mathbb{Z}}}\left\vert x_{n}\right\vert ,\max_{n\in\left[  1,T\right]
_{\mathbb{Z}}}\left\vert y_{n}\right\vert \right\}  ,
\]
where $[1,T]_{\mathbb{Z}}:=\left[  1,T\right]  \cap\mathbb{Z}$. Let us define
the subset $\Omega\left(  W\right)  $ of $P_{T}$ by
\[
\Omega\left(  W\right)  :=\{(x,y)\in P_{T}:\left\Vert (x,y)\right\Vert \leq
W\},
\]
where $W>0$ is a constant. Then $\Omega\left(  W\right)  $ is a bounded,
closed and convex subset of $P_{T}$.

For any pair $\left(  x,y\right)  =\left\{  \left(  x_{n}\left(  \eta\right)
,y_{n}\left(  \zeta\right)  \right)  \right\}  _{n\in\mathbb{Z}}\in
\Omega\left(  W\right)  $ with an initial sequence $\left\{  \left(  \eta
_{n},\zeta_{n}\right)  \right\}  _{n\in\mathbb{Z}^{-}}$ in $S$, define the
mapping $E$ on $\Omega\left(  W\right)  $ by%
\[
E\left(  x,y\right)  :=\left\{  E(x,y)_{n}\right\}  _{n\in\mathbb{Z}%
}:=\left\{  \left(
\begin{array}
[c]{c}%
E_{1}(x,y)_{n}\\
E_{2}(x,y)_{n}%
\end{array}
\right)  \right\}  _{n\in\mathbb{Z}},
\]
where
\begin{equation}
E_{1}(x,y)_{n}:=\left\{
\begin{array}
[c]{ll}%
\eta_{n} & \text{for }n\in\mathbb{Z}^{-}\\
\alpha_{h}\sum\limits_{i=n}^{n+T-1}\left(  \prod\limits_{l=i+1}^{n+T-1}%
(1+h_{l})\right)  \sum\limits_{m=-\infty}^{i}a_{i,m}f(y_{m}) & \text{for }%
n\in\mathbb{N}%
\end{array}
\right.  , \label{1.5}%
\end{equation}%
\begin{equation}
E_{2}(x,y)_{n}:=\left\{
\begin{array}
[c]{ll}%
\zeta_{n} & \text{for }n\in\mathbb{Z}^{-}\\
\alpha_{p}%
{\textstyle\sum\limits_{i=n}^{n+T-1}}
\left(  \prod\limits_{l=i+1}^{n+T-1}(1+p_{l})\right)  \sum\limits_{m=-\infty
}^{i}b_{i,m}g(x_{m}) & \text{for }n\in\mathbb{N}%
\end{array}
\right.  , \label{1.5.1}%
\end{equation}
and%
\begin{align*}
\alpha_{h}  &  :=\left[  1-\prod\limits_{l=0}^{T-1}(1+h_{l})\right]  ^{-1},\\
\alpha_{p}  &  :=\left[  1-\prod\limits_{l=0}^{T-1}(1+p_{l})\right]  ^{-1}.
\end{align*}
We shall use the following result on several occasions in our further analysis.

\begin{lemma}
\label{lem1}Assume that (\ref{1.2}-\ref{1.2.2}) hold. Suppose that
$1+h_{n}\neq0$, $1+p_{n}\neq0$ for all $n\in\lbrack1,T]_{\mathbb{Z}}$, and
that
\begin{equation}
\prod\limits_{l=0}^{T-1}(1+h_{l})\neq1\text{ and }\prod\limits_{l=0}%
^{T-1}(1+p_{l})\neq1. \label{1.3.}%
\end{equation}
The pair$\ \left(  x,y\right)  =\left\{  \left(  x_{n}\left(  \eta\right)
,y_{n}\left(  \zeta\right)  \right)  \right\}  _{n\in\mathbb{Z}}\in
\Omega\left(  W\right)  $ with an initial sequence $\left\{  \left(  \eta
_{n},\zeta_{n}\right)  \right\}  _{n\in\mathbb{Z}^{-}}$ in $S$ satisfies%
\[
E(x,y)_{n}=(x_{n},y_{n})
\]
for all $n\in\mathbb{N}$ if and only if it is a $T$-periodic solution of
(\ref{1.1}).
\end{lemma}

\begin{proof}
One may easily verify that the pair $\left(  x,y\right)  =\left\{  \left(
x_{n}\left(  \eta\right)  ,y_{n}\left(  \zeta\right)  \right)  \right\}
_{n\in\mathbb{Z}}\in\Omega\left(  W\right)  $ satisfying $(x_{n}%
,y_{n})=E(x,y)_{n}$ for all $n\in\mathbb{N}$ is a $T$-periodic solution of the
system (\ref{1.1}). Conversely, suppose that the pair $\left(  x,y\right)
=\left\{  \left(  x_{n}\left(  \eta\right)  ,y_{n}\left(  \zeta\right)
\right)  \right\}  _{n\in\mathbb{Z}}\in\Omega\left(  W\right)  $ is a
$T$--periodic solution of (\ref{1.1}). Multiplying both sides of the first
equation in (\ref{1.1}) with $\left(  \prod\limits_{l=0}^{n}(1+h_{l})\right)
^{-1}$ and taking the summation from $n$ to $n+T-1,$ we obtain%
\[
\sum\limits_{i=n}^{n+T-1}\Delta\left[  x_{i}\left(  \prod\limits_{l=0}%
^{i-1}(1+h_{l})\right)  ^{-1}\right]  =\sum\limits_{i=n}^{n+T-1}\left(
\prod\limits_{l=0}^{i}(1+h_{l})\right)  ^{-1}\sum\limits_{m=-\infty}%
^{i}a_{i,m}f(y_{m}).
\]
This implies that%
\begin{align*}
&  x_{n+T}\left(  \prod\limits_{l=0}^{n+T-1}(1+h_{l})\right)  ^{-1}%
-x_{n}\left(  \prod\limits_{l=0}^{n-1}(1+h_{l})\right)  ^{-1}\\
&  =\sum\limits_{i=n}^{n+T-1}\left(  \prod\limits_{l=0}^{i}(1+h_{l})\right)
^{-1}\sum\limits_{m=-\infty}^{i}a_{i,m}f(y_{m}).
\end{align*}
Using the equalities $x_{n+T}=$ $x_{n}$ and $\prod\limits_{l=n}^{n+T-1}%
(1+h_{l})=\prod\limits_{l=0}^{T-1}(1+h_{l})$, $n\in\mathbb{N}$, we have
$E_{1}(x,y)_{n}=(x_{n},y_{n})$ for all $n\in\mathbb{N}$. The equality
$E_{2}(x,y)_{n}=(x_{n},y_{n})$ for $n\in\mathbb{N}$ can be obtained by using a
similar procedure. The proof is complete.
\end{proof}

In preparation for the next result we assume that there exist positive
constants $W_{1},$ $W_{2},$ $K_{1}$, and $K_{2}$ such that%
\begin{align}
\left\vert f(x)\right\vert  &  \leq W_{1}\label{2.1}\\
\left\vert g(y)\right\vert  &  \leq W_{2} \label{2.2}%
\end{align}%
\begin{equation}
\left\vert \alpha_{h}\right\vert \sum\limits_{i=n}^{n+T-1}\left\vert
\prod\limits_{l=i+1}^{n+T-1}(1+h_{l})\right\vert \sum\limits_{m=-\infty}%
^{i}\left\vert a_{i,m}\right\vert \leq K_{1} \label{2.3}%
\end{equation}%
\begin{equation}
\left\vert \alpha_{p}\right\vert \sum\limits_{i=n}^{n+T-1}\left\vert
\prod\limits_{l=i+1}^{n+T-1}(1+p_{l})\right\vert \sum\limits_{m=-\infty}%
^{i}\left\vert b_{i,m}\right\vert \leq K_{2} \label{2.4}%
\end{equation}
for all $n\in\mathbb{Z}$ and all $\left(  x,y\right)  \in\Omega\left(
W\right)  $.

\begin{theorem}
\label{thm2} In addition to the assumptions of Lemma \ref{lem1} suppose that
(\ref{2.1}-\ref{2.4}) hold. Then (\ref{1.1}) has a $T$-periodic solution.
\end{theorem}

\begin{proof}
From Lemma \ref{lem1}, we can deduce that $E(x,y)_{n+T}=E(x,y)_{n}$ for all
$n\in\mathbb{N}$ and any $(x,y)\in\Omega\left(  W\right)  $. Moreover, if
$(x,y)\in\Omega\left(  W\right)  $ then%
\begin{equation}
\left\vert E_{1}(x,y)_{n}\right\vert \leq\left\vert \alpha_{h}\right\vert
\sum\limits_{i=n}^{n+T-1}\left\vert \prod\limits_{l=i+1}^{n+T-1}%
(1+h_{l})\right\vert \sum\limits_{m=-\infty}^{i}\left\vert a_{i,m}\right\vert
\left\vert f(y_{m})\right\vert \leq W_{1}K_{1}, \label{2.4*}%
\end{equation}
and%
\begin{equation}
\left\vert E_{2}(x,y)_{n}\right\vert \leq\left\vert \alpha_{p}\right\vert
\sum\limits_{i=n}^{n+T-1}\left\vert \alpha_{p}\prod\limits_{l=i+1}%
^{n+T-1}(1+p_{l})\right\vert \sum\limits_{m=-\infty}^{i}\left\vert
b_{i,m}\right\vert \left\vert g(x_{m})\right\vert \leq W_{2}K_{2} \label{2.5*}%
\end{equation}
for all $n\in\mathbb{N}$. If we set $W=\max\{W_{1}K_{1},W_{2}K_{2}\}$ then $E$
maps $\Omega\left(  W\right)  $ into itself. Now we show that $E$ is
continuous. Let $\{(x^{l},y^{l})\}$, $l\in\mathbb{N}$, be a sequence in
$\Omega\left(  W\right)  $ such that%
\begin{align*}
\lim_{l\rightarrow\infty}\left\Vert (x^{l},y^{l})-(x,y)\right\Vert  &
=\lim_{l\rightarrow\infty}\left(  \max_{n\in\left[  1,T\right]  _{\mathbb{Z}}%
}\left\{  \left\vert x_{n}^{l}-x_{n}\right\vert ,\left\vert y_{n}^{l}%
-y_{n}\right\vert \right\}  \right) \\
&  =0.
\end{align*}
Since $\Omega\left(  W\right)  $ is closed, we must have $(x,y)\in
\Omega\left(  W\right)  $. Then by definition of $E$ we have
\begin{align*}
\left\Vert E(x^{l},y^{l})-E(x,y)\right\Vert  &  =\max\left\{  \max
_{n\in\left[  1,T\right]  _{\mathbb{Z}}}\left\vert E_{1}(x^{l},y^{l}%
)_{n}-E_{1}(x,y)_{n}\right\vert ,\right. \\
&  \left.  \max_{n\in\left[  1,T\right]  _{\mathbb{Z}}}\left\vert E_{2}%
(x^{l},y^{l})_{n}-E_{2}(x,y)_{n}\right\vert \right\}  ,
\end{align*}
in which%
\begin{align*}
\left\vert E_{1}(x^{l},y^{l})_{n}-E_{1}(x,y)_{n}\right\vert  &  =\left\vert
\alpha_{h}\right\vert \left\vert \sum\limits_{i=n}^{n+T-1}\left(
\prod\limits_{l=i+1}^{n+T-1}(1+h_{l})\right)  \sum\limits_{m=-\infty}%
^{i}a_{i,m}f(y_{m}^{l})-\right. \\
&  \left.  \sum\limits_{i=n}^{n+T-1}\left(  \prod\limits_{l=i+1}%
^{n+T-1}(1+h_{l})\right)  \sum\limits_{m=-\infty}^{i}a_{i,m}f(y_{m}%
)\right\vert \\
&  \leq\left\vert \alpha_{h}\right\vert \sum\limits_{i=n}^{n+T-1}\left\vert
\prod\limits_{l=i+1}^{n+T-1}(1+h_{l})\right\vert \sum\limits_{m=-\infty}%
^{i}\left\vert a_{i,m}\right\vert \left\vert f(y_{m}^{l})-f(y_{m})\right\vert
.
\end{align*}
Similarly,%
\[
\left\vert E_{2}(x^{l},y^{l})_{n}-E_{2}(x,y)_{n}\right\vert \leq\left\vert
\alpha_{p}\right\vert \sum\limits_{i=n}^{n+T-1}\left\vert \prod\limits_{l=i+1}%
^{n+T-1}(1+p_{l})\right\vert \sum\limits_{m=-\infty}^{i}\left\vert
b_{i,m}\right\vert \left\vert g(x_{m}^{l})-g(x_{m})\right\vert .
\]
The continuity of $f$ and $g$ along with the Lebesgue dominated convergence
theorem imply that%
\[
\lim_{l\rightarrow\infty}\left\Vert E(x^{l},y^{l})-E(x,y)\right\Vert =0.
\]
This shows that $E$ is continuous. Finally, we have to show that
$E\Omega\left(  W\right)  $ is precompact. Let $\{(x^{l},y^{l})\}_{l\in
\mathbb{N}}$ be a sequence in $\Omega\left(  W\right)  $. For each fixed
$l\in\mathbb{N}$, $\{(x_{n}^{l},y_{n}^{l})\}_{n\in\mathbb{Z}}$ is a bounded
sequence of real pairs. Then by \emph{Bolzano-Weierstrass Theorem,}
$\{(x_{n}^{l},y_{n}^{l})\}_{n\in\mathbb{Z}}$ has a convergent subsequence
$\{(x_{n_{k}}^{l},y_{n_{k}}^{l})\}$. By repeating the diagonalization process
for each $l\in\mathbb{N},$ we can construct a convergent subsequence
$\{(x^{l_{k}},y^{l_{k}})\}_{l_{k}\in\mathbb{N}}$ of $\{(x^{l},y^{l}%
)\}_{l\in\mathbb{N}}$ in $\Omega\left(  W\right)  .$ Since $E$ is continuous,
we deduce that $\{E(x^{l},y^{l})\}_{l\in\mathbb{N}}$ has a convergent
subsequence in $E\Omega\left(  W\right)  $. This means, $E\Omega\left(
W\right)  $ is precompact. By Schauder's fixed point theorem we conclude that
there exists a pair $(x,y)\in\Omega\left(  W\right)  $ such that
$E(x,y)=(x,y)$.
\end{proof}

\begin{theorem}
\label{thm3}In addition to the assumptions of Lemma \ref{lem1}, we assume that
(\ref{2.1}), (\ref{2.3}) and (\ref{2.4}) hold. If $g$ is a non-decreasing
function satisfying%
\begin{equation}
\left\vert g(x)\right\vert \leq g(\left\vert x\right\vert ), \label{2.6*}%
\end{equation}
then (\ref{1.1}) has a $T$-periodic solution.
\end{theorem}

\begin{proof}
By (\ref{2.3}) and (\ref{2.4*}) we already have%
\[
\left\vert E_{1}(x,y)\right\vert \leq W_{1}K_{1}\text{ for all }(x,y)\in
\Omega\left(  W\right)  .
\]
This along with (\ref{2.6*}) imply%
\begin{align*}
\left\vert E_{2}(x,y)_{n}\right\vert  &  \leq\sum\limits_{i=n}^{n+T-1}%
\left\vert \alpha_{p}\prod\limits_{l=i+1}^{n+T-1}(1+p_{l})\right\vert
\sum\limits_{m=-\infty}^{i}\left\vert b_{i,m}\right\vert \left\vert
g(x_{m})\right\vert \\
&  \leq\sum\limits_{i=n}^{n+T-1}\left\vert \alpha_{p}\prod\limits_{l=i+1}%
^{n+T-1}(1+p_{l})\right\vert \sum\limits_{m=-\infty}^{i}\left\vert
b_{i,m}\right\vert g(\left\vert E_{1}(x,y)\right\vert )\\
&  \leq K_{2}g(W_{1}K_{1}).
\end{align*}
If we set $W=\max\{W_{1}K_{1},K_{2}g(W_{1}K_{1})\}$, then the rest of the
proof is similar to the proof of Theorem \ref{thm2} and hence we omit it.
\end{proof}

Similarly, we can give the following result.

\begin{theorem}
\label{thm4}In addition to the assumptions of Lemma \ref{lem1}, we assume
(\ref{2.2}), (\ref{2.3}) and (\ref{2.4}) hold. If $f$ is a non-decreasing
function satisfying%
\[
\left\vert f(y)\right\vert \leq f(\left\vert y\right\vert ),
\]
then (\ref{1.1}) has a $T$-periodic solution.
\end{theorem}

\begin{example}
\label{Examp1}Let
\begin{align*}
h_{n}  &  =1+\cos n\pi,\\
p_{n}  &  =1-\cos n\pi,\\
a_{n,i}  &  =b_{n,i}=e^{i-n},
\end{align*}
and%
\[
f(x)=\sin x\text{ and }g(x)=\sin2x.
\]
Then (\ref{1.1}) turns into the following system%
\[
\left\{
\begin{array}
[c]{c}%
\Delta x_{n}=(1+\cos n\pi)x_{n}+\sum\limits_{i=-\infty}^{n}e^{i-n}\sin
(y_{i}),\\
\Delta y_{n}=(1-\cos n\pi)y_{n}+\sum\limits_{i=-\infty}^{n}e^{i-n}\sin
(2x_{i})\text{\ \ }%
\end{array}
\right.  .
\]
It can be easily verified that conditions (\ref{1.2}-\ref{1.3.}) and
(\ref{2.1}-\ref{2.4}) hold. By Theorem \ref{thm2}, there exists a $2$-periodic
solution $(x,y)=\left\{  (x_{n},y_{n})\right\}  _{ne\mathbb{Z}}$ of system
(\ref{1.1}) satisfying%
\begin{align*}
x_{n}  &  =-\frac{1}{2}\sum\limits_{i=n}^{n+1}\prod\limits_{l=i+1}%
^{n+1}(2+\cos(l\pi))\sum\limits_{m=-\infty}^{i}e^{m-i}\sin(y_{m})\text{,}\\
y_{n}  &  =-\frac{1}{2}\sum\limits_{i=n}^{n+1}\prod\limits_{l=i+1}%
^{n+1}(2-\cos(l\pi))\sum\limits_{m=-\infty}^{i}e^{m-i}\sin(2x_{m})\text{,}%
\end{align*}
for all $n\in\mathbb{N}$.
\end{example}

\section{Asymptotic Periodicity}

In this section, we are going to show the existence of an asymptotically
$T$-periodic solution of system (\ref{1.1}) by using Schauder's fixed point
theorem. First we state the following definition.

\begin{definition}
\label{def2}A sequence $\left\{  x_{n}\right\}  _{n\in\mathbb{Z}}$ is called
asymptotically $T$-periodic if there exist two sequences $u_{n}$ and $v_{n}$
such that $u_{n}$ is $T$-periodic, $\lim_{n\rightarrow\infty}v_{n}=0$, and
$x_{n}=u_{n}+v_{n}$ for all $n\in\mathbb{Z}.$
\end{definition}

First, we suppose that
\begin{equation}
\prod\limits_{j=0}^{T-1}(1+h_{j})=1\text{ and }\prod\limits_{j=0}%
^{T-1}(1+p_{j})=1. \label{3.2.1}%
\end{equation}
Then we define the sequences $\varphi:=\{\varphi_{n}\}_{n\in\mathbb{N}}$ and
$\psi:=\{\psi_{n}\}_{n\in\mathbb{N}}$ as follows%
\begin{equation}
\varphi_{n}:=\prod\limits_{j=0}^{n-1}\frac{1}{1+h_{j}}\text{ and }\psi
_{n}:=\prod\limits_{j=0}^{n-1}\frac{1}{1+p_{j}}. \label{3.2}%
\end{equation}
Furthermore, we define the constants $m_{k}$, $M_{k}$, $k=1,2$, by%
\[
m_{1}:=\min_{i\in\lbrack1,T]_{\mathbb{Z}}}\left\vert \varphi_{i}\right\vert
,\ \ M_{1}:=\max_{i\in\lbrack1,T]_{\mathbb{Z}}}\left\vert \varphi
_{i}\right\vert ,\ m_{2}:=\min_{i\in\lbrack1,T]_{\mathbb{Z}}}\left\vert
\psi_{i}\right\vert ,\ \ M_{2}:=\max_{i\in\lbrack1,T]_{\mathbb{Z}}}\left\vert
\psi_{i}\right\vert .
\]

We note that in this section, we do not assume (\ref{1.2.1}-\ref{1.2.2}) but
instead we ask that the series%
\begin{equation}
\sum\limits_{i=0}^{\infty}\sum\limits_{m=-\infty}^{i}\left\vert a_{i,m}%
\right\vert <\infty\text{ and }\sum\limits_{i=0}^{\infty}\sum
\limits_{m=-\infty}^{i}\left\vert b_{i,m}\right\vert <\infty\label{3.3}%
\end{equation}
converge to $a$ and $b$, respectively. Observe that (\ref{3.3}) implies%
\begin{equation}
\lim_{n\rightarrow\infty}\sum\limits_{i=n}^{\infty}\sum\limits_{m=-\infty}%
^{i}\left\vert a_{i,m}\right\vert =\lim_{n\rightarrow\infty}\sum
\limits_{i=n}^{\infty}\sum\limits_{m=-\infty}^{i}\left\vert b_{i,m}\right\vert
=0\text{.} \label{3.3.1}%
\end{equation}

\begin{theorem}
\label{thm6}Suppose that (\ref{2.1}-\ref{2.2}), (\ref{3.2.1}), and
(\ref{3.3}-\ref{3.3.1}) hold. Then system (\ref{1.1}) has an asymptotically
$T$-periodic solution $(x,y)=\left\{  \left(  x_{n},y_{n}\right)  \right\}
_{n\in\mathbb{Z}}$ satisfying%
\begin{align*}
x_{n}  &  :=u_{n}^{(1)}+v_{n}^{(1)}\\
y_{n}  &  :=u_{n}^{(2)}+v_{n}^{(2)}%
\end{align*}
for $n\in\mathbb{N}$, where
\[
u_{n}^{(1)}=c_{1}\prod\limits_{j=0}^{n-1}(1+h_{j}),\ \ \ \ u_{n}^{(2)}%
=c_{2}\prod\limits_{j=0}^{n-1}(1+p_{j})\text{, }n\in\mathbb{Z}^{+}%
\]
$c_{1}$ and $c_{2}$ are positive constants, and
\[
\lim_{n\rightarrow\infty}v_{n}^{(1)}=\lim_{n\rightarrow\infty}v_{n}^{(2)}=0.
\]

\end{theorem}

\begin{proof}
Due to the $T$-periodicity of the sequences $\left\{  h_{n}\right\}
_{n\in\mathbb{Z}}$ and $\left\{  p_{n}\right\}  _{n\in\mathbb{Z}}$ and by
(\ref{3.2.1}-\ref{3.2}) we have
\[
\varphi_{n}\in\{\varphi_{1},\varphi_{2},...,\varphi_{T}\}\text{ and }\psi
_{n}\in\{\psi_{1},\psi_{2},...,\psi_{T}\}
\]
for all $n\in\mathbb{N}.$ This means%
\begin{align}
m_{1}  &  \leq\left\vert \varphi_{n}\right\vert \leq M_{1}\label{3.4}\\
m_{2}  &  \leq\left\vert \psi_{n}\right\vert \leq M_{2} \label{3.5}%
\end{align}
for all $n\in\mathbb{N}$. Let $B$ be the set of all real bounded sequences
$x=\left\{  x_{n}\right\}  _{ne\mathbb{Z}}$. Denote by $\mathbb{B}$ the Banach
space of all pairs of real bounded sequences $\left(  x,y\right)
=\{(x_{n},y_{n})\}_{n\in\mathbb{Z}}$, $x,y\in B$, endowed with the maximum
norm
\[
\left\Vert \left(  x,y\right)  \right\Vert =\max\{\sup_{n\in\mathbb{Z}%
}\left\vert x_{n}\right\vert ,\sup_{n\in\mathbb{Z}}\left\vert y_{n}\right\vert
\}.
\]

For a positive constant $W^{\ast}$ we define
\[
\Omega^{\ast}\left(  W^{\ast}\right)  :=\{(x,y)\in\mathbb{B}:\left\Vert
(x,y)\right\Vert \leq W^{\ast}\}.
\]
Then, $\Omega^{\ast}\left(  W^{\ast}\right)  $ is a nonempty bounded convex,
and closed subset of $\mathbb{B}.$ For any pair \newline$\left(  x,y\right)
=\left\{  \left(  x_{n}\left(  \eta\right)  ,y_{n}\left(  \zeta\right)
\right)  \right\}  _{n\in\mathbb{Z}}\in\Omega^{\ast}\left(  W^{\ast}\right)  $
with an initial sequence $\left\{  \left(  \eta_{n},\zeta_{n}\right)
\right\}  _{n\in\mathbb{Z}^{-}}$ in $S$ define the mapping $E^{\ast}$ on
$\Omega^{\ast}\left(  W^{\ast}\right)  $ by%
\[
E^{\ast}\left(  x,y\right)  =\left\{  E^{\ast}\left(  x,y\right)
_{n}\right\}  _{n\in\mathbb{Z}}=\left\{  \left(
\begin{array}
[c]{c}%
E_{1}^{\ast}(x,y)_{n}\\
E_{2}^{\ast}(x,y)_{n}%
\end{array}
\right)  \right\}  _{n\in\mathbb{Z}},
\]
where%
\begin{equation}
E_{1}^{\ast}(x,y)_{n}:=\left\{
\begin{array}
[c]{ll}%
\eta_{n} & \text{for }n\in\mathbb{Z}^{-}\\
c_{1}\frac{1}{\varphi_{n}}-\sum\limits_{i=n}^{\infty}\sum\limits_{m=-\infty
}^{i}\frac{\varphi_{i+1}}{\varphi_{n}}a_{i,m}f(y_{m}) & \text{for }%
n\in\mathbb{N}%
\end{array}
\right.  \text{,} \label{3.8}%
\end{equation}
and%
\begin{equation}
E_{2}^{\ast}(x,y)_{n}:=\left\{
\begin{array}
[c]{ll}%
\zeta_{n} & \text{for }n\in\mathbb{Z}^{-}\\
c_{2}\frac{1}{\psi_{n}}-\sum\limits_{i=n}^{\infty}\sum\limits_{m=-\infty}%
^{i}\frac{\psi_{i+1}}{\psi_{n}}b_{i,m}g(x_{m}) & \text{for }n\in\mathbb{N}%
\end{array}
\right.  . \label{3.9}%
\end{equation}
We will show that the mapping $E^{\ast}$ has a fixed point in $\mathbb{B}$.
First, we demonstrate that $E^{\ast}\Omega^{\ast}\left(  W^{\ast}\right)
\subset\Omega^{\ast}\left(  W^{\ast}\right)  $. If $(x,y)\in\Omega^{\ast
}\left(  W^{\ast}\right)  $, then%
\begin{align}
\left\vert E_{1}^{\ast}(x,y)_{n}-c_{1}\frac{1}{\varphi_{n}}\right\vert  &
\leq M_{1}m_{1}^{-1}W_{1}\sum\limits_{i=n}^{\infty}\sum\limits_{m=-\infty}%
^{i}\left\vert a_{i,m}\right\vert \label{3.10}\\
&  \leq M_{1}m_{1}^{-1}W_{1}\sum\limits_{i=0}^{\infty}\sum\limits_{m=-\infty
}^{i}\left\vert a_{i,m}\right\vert \nonumber\\
&  =M_{1}m_{1}^{-1}W_{1}a,
\end{align}
and%
\begin{align}
\left\vert E_{2}^{\ast}(x,y)_{n}-c_{2}\frac{1}{\psi_{n}}\right\vert  &  \leq
M_{2}m_{2}^{-1}W_{2}\sum\limits_{i=n}^{\infty}\sum\limits_{m=-\infty}%
^{i}\left\vert b_{i,m}\right\vert \label{3.11}\\
&  \leq M_{2}m_{2}^{-1}W_{2}\sum\limits_{i=0}^{\infty}\sum\limits_{m=-\infty
}^{i}\left\vert b_{i,m}\right\vert \nonumber\\
&  =M_{2}m_{2}^{-1}W_{2}b
\end{align}
for all $n\in\mathbb{N}$. This implies that%
\[
\left\vert E_{1}^{\ast}(x,y)_{n}\right\vert \leq M_{1}m_{1}^{-1}W_{1}%
a+\frac{c_{1}}{m_{1}}%
\]
and%
\[
\left\vert E_{2}^{\ast}(x,y)_{n}\right\vert \leq M_{2}m_{2}^{-1}W_{2}%
b+\frac{c_{2}}{m_{2}}%
\]
for all $n\in\mathbb{N}$. If we set
\[
W^{\ast}=\max\{M_{1}m_{1}^{-1}W_{1}a+\frac{c_{1}}{m_{1}},M_{2}m_{2}^{-1}%
W_{2}b+\frac{c_{2}}{m_{2}}\},
\]
then \ we have $E^{\ast}\Omega^{\ast}\left(  W^{\ast}\right)  \subset
\Omega^{\ast}\left(  W^{\ast}\right)  $ as desired. \newline Next, we show
that $E^{\ast}$ is continuous. Let $\{(x^{q},y^{q})\}_{q\in\mathbb{N}}$ be a
sequence in $\Omega^{\ast}\left(  W^{\ast}\right)  $ such that \newline%
$\lim_{q\rightarrow\infty}\left\Vert (x^{q},y^{q})-(x,y)\right\Vert =0$, where
$\left(  x,y\right)  =\left\{  \left(  x_{n},y_{n}\right)  \right\}
_{n\in\mathbb{Z}}$. Since $\Omega^{\ast}\left(  W^{\ast}\right)  $ is closed,
we must have $(x,y)\in\Omega^{\ast}\left(  W^{\ast}\right)  $. From
(\ref{3.8}) and (\ref{3.9}), we have%
\[
\left\vert E_{1}^{\ast}(x^{q},y^{q})_{n}-E_{1}^{\ast}(x,y)_{n}\right\vert
\leq\sum\limits_{i=n}^{\infty}\sum\limits_{m=-\infty}^{i}\left\vert
\frac{\varphi_{i+1}}{\varphi_{n}}\right\vert \left\vert a_{i,m}\right\vert
\left\vert f(y_{m}^{q})-f(y_{m})\right\vert
\]
and%
\[
\left\vert E_{2}^{\ast}(x^{q},y^{q})_{n}-E_{2}^{\ast}(x,y)_{n}\right\vert
\leq\sum\limits_{i=n}^{\infty}\sum\limits_{m=-\infty}^{i}\left\vert \frac
{\psi_{i+1}}{\psi_{n}}\right\vert \left\vert b_{i,m}\right\vert \left\vert
g(x_{m}^{q})-g(x_{m})\right\vert
\]
for all $n\in\mathbb{N}$. Since $f$ and $g$ are continuous, we have by the
Lebesgue dominated convergence theorem that%
\[
\lim_{q\rightarrow\infty}\left\Vert E^{\ast}(x^{q},y^{q})-E^{\ast
}(x,y)\right\Vert =0.
\]
As we did in the proof of Theorem \ref{thm2} we can show that $E^{\ast}$ has a
fixed point in $\Omega^{\ast}\left(  W^{\ast}\right)  $. On the other hand,
using a similar procedure that we have employed in the proof of Lemma
\ref{lem1}, we can deduce that any solution $(x,y)=\left\{  \left(
x_{n},y_{n}\right)  \right\}  _{n\in\mathbb{Z}}$ of the system (\ref{1.1}) is
a fixed point for the operator $E^{\ast}$. This means $E^{\ast}\left(
x,y\right)  =\left(  x,y\right)  $ or equivalently,%
\begin{equation}
x_{n}=c_{1}\frac{1}{\varphi_{n}}-\sum\limits_{i=n}^{\infty}\sum
\limits_{m=-\infty}^{i}\frac{\varphi_{i+1}}{\varphi_{n}}a_{i,m}f(y_{m})
\label{3.12}%
\end{equation}
and%
\begin{equation}
y_{n}=c_{2}\frac{1}{\psi_{n}}-\sum\limits_{i=n}^{\infty}\sum\limits_{m=-\infty
}^{i}\frac{\psi_{i+1}}{\psi_{n}}b_{i,m}g(x_{m}). \label{3.13}%
\end{equation}
Conversely, any pair $(x,y)=\left\{  \left(  x_{n},y_{n}\right)  \right\}
_{n\in\mathbb{Z}}$ satisfying (\ref{3.12}) and (\ref{3.13}) will also satisfy%
\begin{align*}
x_{n+1}-x_{n}(1+h_{n})  &  =c_{1}(\prod\limits_{j=0}^{n}(1+h_{j}%
)-(1+h_{n})\prod\limits_{j=0}^{n-1}(1+h_{j}))\\
&  +(1+h_{n})\sum\limits_{i=n}^{\infty}\sum\limits_{m=-\infty}^{i}%
\frac{\varphi_{i+1}}{\varphi_{n}}a_{i,m}f(y_{m})\\
&  -\sum\limits_{i=n+1}^{\infty}\sum\limits_{m=-\infty}^{i}\frac{\varphi
_{i+1}}{\varphi_{n+1}}a_{i,m}f(y_{m}),
\end{align*}
and hence,%
\begin{align*}
x_{n+1}-x_{n}(1+h_{n})  &  =\sum\limits_{i=n}^{\infty}\sum\limits_{m=-\infty
}^{i}\frac{(1+h_{n})\prod\limits_{j=0}^{n-1}(1+h_{j})}{\prod\limits_{j=0}%
^{i}(1+h_{j})}a_{i,m}f(y_{m})\\
&  -\sum\limits_{i=n+1}^{\infty}\sum\limits_{m=-\infty}^{i}\frac
{\prod\limits_{j=0}^{n}(1+h_{j})}{\prod\limits_{j=0}^{i}(1+h_{j})}%
a_{i,m}f(y_{m})\\
&  =\sum\limits_{m=-\infty}^{n}a_{n,m}f(y_{m}).
\end{align*}
That is, any fixed point $\left(  x,y\right)  =\left\{  (x_{n},y_{n})\right\}
_{n\in\mathbb{Z}}$ of the operator $E^{\ast}$ satisfies the first equation in
(\ref{1.1}). Similarly, one may show that the second equation holds. \newline
For an arbitrary fixed point $(x,y)\in\Omega^{\ast}\left(  W^{\ast}\right)  $
of $E^{\ast}$, we have%
\begin{equation}
\lim_{n\rightarrow\infty}\left\vert x_{n}-c_{1}\frac{1}{\varphi_{n}%
}\right\vert =\lim_{n\rightarrow\infty}\left\vert E_{1}^{\ast}(x,y)_{n}%
-c_{1}\frac{1}{\varphi_{n}}\right\vert =0 \label{3.13*}%
\end{equation}
and%
\begin{equation}
\lim_{n\rightarrow\infty}\left\vert y_{n}-c_{2}\frac{1}{\psi_{n}}\right\vert
=\lim_{n\rightarrow\infty}\left\vert E_{2}(x,y)_{n}-c_{2}\frac{1}{\psi_{n}%
}\right\vert =0. \label{3.13**}%
\end{equation}
Choosing%
\begin{equation}
u_{n}^{(1)}=c_{1}\frac{1}{\varphi_{n}},\ \ v_{n}^{(1)}=-\sum\limits_{i=n}%
^{\infty}\sum\limits_{m=-\infty}^{i}\frac{\varphi_{i+1}}{\varphi_{n}}%
a_{i,m}f(y_{m}) \label{3.13.1}%
\end{equation}
and%
\begin{equation}
u_{n}^{(2)}=c_{2}\frac{1}{\psi_{n}},\ \ v_{n}^{(2)}=-\sum\limits_{i=n}%
^{\infty}\sum\limits_{m=-\infty}^{i}\frac{\psi_{i+1}}{\psi_{n}}b_{i,m}%
g(x_{m}), \label{3.13.2}%
\end{equation}
we have $x_{n}=u_{n}^{(1)}+v_{n}^{(1)}$ and $y_{n}=u_{n}^{(2)}+v_{n}^{(2)}$.
By (\ref{3.13*}) and (\ref{3.13**}), $v_{n}^{(1)}$ and $v_{n}^{(2)}$ tend to
$0$ when $n\rightarrow\infty.$ Left to show that $u_{n}^{(1)}$ and
$u_{n}^{(2)}$ are $T$-periodic.
\begin{align*}
u_{n+T}^{(1)}  &  =c_{1}\prod\limits_{j=0}^{n+T-1}(1+h_{j})=c_{1}%
\prod\limits_{j=0}^{n-1}(1+h_{j})\prod\limits_{j=n}^{n+T-1}(1+h_{j})\\
&  =c_{1}\prod\limits_{j=0}^{n-1}(1+h_{j})\prod\limits_{j=0}^{T-1}(1+h_{j})\\
&  =c_{1}\prod\limits_{j=0}^{n-1}(1+h_{j}),\;\mbox{by}\;\eqref{3.2.1}.
\end{align*}
Proof for $u_{n}^{(2)}$ is identical.
\end{proof}

\begin{example}
\label{examp2}Consider the system (\ref{1.1}) with the following entries%
\begin{align*}
h_{n}  &  =p_{n}=\left\{
\begin{array}
[c]{cc}%
1, & \text{if }n=2k+1\text{ for }k\in\mathbb{Z}\text{ }\\
-\frac{1}{2}, & \text{if }n=2k\text{ for }k\in\mathbb{Z}%
\end{array}
\right.  ,\\
a_{n,i}  &  =e^{i-2n}\text{, for }n,i\in\mathbb{Z}\\
b_{n,i}  &  =e^{2i-3n}\text{,}\ \ \text{for }n,i\in\mathbb{Z}\\
f(x)  &  =\cos x\text{ and }g(x)=\cos2x.
\end{align*}
Then (\ref{1.1}) turns into the following system:%
\[
\left\{
\begin{array}
[c]{c}%
\Delta x_{n}=h_{n}x_{n}+\sum\limits_{i=-\infty}^{n}e^{i-2n}\cos(y_{i}),\\
\Delta y_{n}=p_{n}y_{n}+\sum\limits_{i=-\infty}^{n}e^{2i-3n}\cos
(2x_{i})\text{\ \ }%
\end{array}
\right.  .
\]
Obviously, the sequences $\left\{  h_{n}\right\}  _{n\in\mathbb{Z}}$ and
$\left\{  p_{n}\right\}  _{n\in\mathbb{Z}}$ are $2-$periodic and all
conditions of Theorem \ref{thm6} are satisfied. Hence, we conclude by Theorem
\ref{thm6} the existence of an asymptotically $2-$periodic solution $\left(
x,y\right)  =\left\{  \left(  x_{n},y_{n}\right)  \right\}  _{n\in\mathbb{Z}}$
satisfying%
\begin{align*}
x_{n}  &  =c_{1}\frac{1}{\varphi_{n}}-\sum\limits_{i=n}^{\infty}%
\sum\limits_{m=-\infty}^{i}\frac{\varphi_{i+1}}{\varphi_{n}}e^{m-2i}\cos
(y_{m})\\
y_{n}  &  =c_{2}\frac{1}{\psi_{n}}-\sum\limits_{i=n}^{\infty}\sum
\limits_{m=-\infty}^{i}\frac{\psi_{i+1}}{\psi_{n}}e^{2m-3i}\cos(2x_{m}),
\end{align*}
for all $n\in\mathbb{N}$, where $c_{1}$ and $c_{2}$ are positive constants,
$\varphi:=\{\varphi_{n}\}_{n\in\mathbb{N}}$ and $\psi:=\{\psi_{n}%
\}_{n\in\mathbb{N}}$ are as in (\ref{3.2}).
\end{example}

\end{document}